\documentclass[onecolumn, 11pt]{article}
\usepackage[top=.75in, bottom=2cm, left=2cm, right=2cm]{geometry}
\usepackage{amsmath,amsfonts,amscd,amssymb,amsthm,bbm,bm}
\usepackage{graphicx}
\usepackage{epstopdf}
\usepackage{overpic}
\usepackage{cancel}
\usepackage{rotating}
\usepackage{url}
\usepackage{caption}
\usepackage{color}
\usepackage{rotating}
\usepackage{multirow}
\usepackage{wrapfig}
\usepackage{mathtools}
\usepackage{subeqnarray}
\usepackage{setspace}
\usepackage{pxfonts}
\usepackage[varbb]{newpxmath} 
\usepackage{newpxtext}
\usepackage{microtype, booktabs}
\usepackage{enumitem}
\usepackage[numbers,sort&compress]{natbib}

\usepackage[bottom,flushmargin,hang,multiple]{footmisc}
\usepackage{lipsum}
\newcommand\blfootnote[1]{%
  \begingroup
  \renewcommand\thefootnote{}\footnote{#1}%
  \addtocounter{footnote}{-1}%
  \endgroup
}

\definecolor{header1}{cmyk}{0,0,0,1}

\def \rank {\operatorname{rank}}
\def \id {\operatorname{I}}

\def \gl {\operatorname{GL}}

\def \ess {\operatorname{ess}}
\def \card {\operatorname{card}}
\def \cen {\operatorname{Z}}
\def \chow {\operatorname{chow}}
\def \orb {\operatorname{orb}}

\numberwithin{equation}{section}
\numberwithin{table}{section}
\numberwithin{equation}{section}
\theoremstyle{plain}
\newtheorem{theorem}{Theorem}[section]
\newtheorem*{theorema}{Theorem A}
\newtheorem*{theoremb}{Theorem B}
\newtheorem*{theoremc}{Theorem C}
\newtheorem{lemma}[theorem]{Lemma}
\newtheorem{proposition}[theorem]{Proposition}
\newtheorem{corollary}[theorem]{Corollary}
\theoremstyle{definition}
\newtheorem{definition}[theorem]{Definition}
\newtheorem{example}[theorem]{Example}
\newtheorem{algorithm}[theorem]{Algorithm}
\newtheorem{remark}[theorem]{Remark}

\newcommand{\diag}{\operatorname{diag}}

\usepackage[utf8]{inputenc}

\usepackage[normalem]{ulem}
\usepackage{color}

\everymath{\displaystyle}
\title{\vspace{-.125in}{\huge\selectfont \textbf{Direct Sum and Direct Product Decompositions \\ of Multivariate Functions}}\vspace{-.075in}}

\author{\normalsize{Hua-Lin Huang$^{1*}$, Yiming Liu$^{1}$, Jianhua Xiang$^{1}$ and Yu Ye$^{2}$}\\
\footnotesize{$^1$ School of Mathematical Sciences, Huaqiao University, Quanzhou 362021, China} \\
\footnotesize{$^{2}$ School of Mathematical Sciences, University of Science and Technology of China, Hefei 230026, China}
}

\date{}

\begin{document}
\maketitle



\blfootnote{$^*$ Corresponding author: hualin.huang@hqu.edu.cn}
\vspace{-.2in}
\begin{abstract}
This paper addresses the problem of whether or not a vector-valued multivariate functions can be expressed as a sum or a product of vector-valued functions in disjoint sets of variables through a proper invertible linear change of variables. The crux is an invariant algebra, the so-called center, that we introduce for a set of multivariate functions with second order partial derivatives. We thus provide simple criteria and algorithms for simultaneous additive and multiplicative decompositions of any set of multivariate functions with minor analytic conditions. This is applied to the factorization problem of multivariate homogeneous polynomials, in particular those that are products of linear forms.  
\end{abstract}



\section{Introduction}\label{sec:intro}

\par To simplify a set of multivariate functions simultaneously, the first thing one might think of is to decompose them either additively or multiplicatively via a suitable change of variables. This helps to express a complex system as a sum or a product of several simpler ones by dimension reduction. It is essentially about simultaneous separation of variables for a set of multivariate functions, which is ubiquitous in mathematics, sciences and engineering. For the simplicity of exposition, throughout the paper we work on complex functions only. This can be easily extended to any functions with appropriate analytic conditions.   

\subsection{Problem statement}
\par In order to state the problem more precisely, we introduce some necessary definitions and notations. For the convenience of the exposition, we regard a set of multivariate functions as a vector-valued function, say $\mathbf{f}(\mathbf{z}) = \left(f_1(\mathbf{z}), \cdots, f_m(\mathbf{z})\right)$. The main problem is to determine whether there exists an invertible linear substitution $\mathbf{z} = A \mathbf{w}$ such that $\mathbf{f}$ can be decomposed as an element-wise sum or product of two or more functions in disjoint sets of variables:
$$
    \mathbf{f}(\mathbf{z}) = \tilde{\mathbf{f}}(A\mathbf{w}) = \tilde{\mathbf{f}}_1(\mathbf{w}_1) \circledast \tilde{\mathbf{f}}_2(\mathbf{w}_2).
$$
Here, $\mathbf{w}$ is the disjoint union of two nonempty subsets $\mathbf{w}_1$ and $\mathbf{w}_2$, and the notation $\circledast$ can be either element-wise addition $+$ or multiplication $\odot$. If such a substitution exists, the function $\mathbf{f}$ is said to be a direct sum or a direct product. Direct sum decomposition may simplify a function by eliminating redundant variables and cross terms, whereas direct product decomposition is a special kind of factorization in which variables involving in distinct factors are disjoint.
\subsection{Related work}
\par  The direct sum decompositions of a single homogeneous polynomial have been considered in many papers, see e.g. \cite{Wer, Fed, Huang-2}. Very recently, the first author and his collaborators tackled the problem of simultaneous direct sum decompositions for a set of homogeneous polynomial \cite{Fang} which generalized the center theory of a homogeneous polynomial \cite{Har-1, Huang-2}. There are relatively less work on direct products of polynomials. Koiran and Ressayre \cite{Koi} provided an algorithm to decompose a homogeneous polynomial into a product of linear forms. Fedorchuk \cite{Fed} interpreted the direct sum decomposability of a homogeneous polynomial in terms of the direct product decomposability of its associated form \cite{Jar}. Direct sum and direct product functions also appear in other areas such as complex analysis, algebraic geometry, and various branches of computational mathematics. For example, Liu and Ren proved that biholomorphic convex mappings of several complex variables on the polydisc are the direct sum of biholomorphic functions of one complex variable \cite{Liu}, products of linear forms constitutes the famous Chow variety of polynomials \cite{Gel}, direct sum functions are ideal approximation for high-dimensional functions \cite{Ba, Kuo} in modelling complex systems.

\subsection{Methods and results}
\par The theory of centers provides an effective approach to the direct sum decomposition problem of polynomials \cite{Huang, Huang-2, Fang}. The present paper arises from the question that we asked ourselves: can we adapt the center theory of polynomials to treat the direct product problem of polynomials? It is natural to observe that direct product polynomials become direct sum functions after taking logarithm, however not polynomials any more. Therefore, the crux is whether we can extend the theory of centers to the more general multivariate functions. 

Recall that the center of a set of polynomials $\{f_1, \cdots, f_m\}$ is defined as
$$
    \cen(f_1, \cdots, f_m) = \left\{
        X \in \mathbb{C}^{n \times n} : 
        (H(f_k)X)^t = H(f_k)X,~ k = 1, \cdots, m
    \right\},
$$
where $H(f_k)$ is the Hessian matrix of $f_k$. Fortunately, we can naturally extend this definition of centers from a set of polynomials to a set of functions in which each admits a Hessian matrix. Moreover, we show that there is a bijection between its direct sum decompositions and the orthogonal idempotent decomposition of the identity of its center. We thus provide a criterion and an algorithm for the simultaneous direct sum decompositions for a set of multivariate functions.

\begin{theorema}[see Theorem \ref{theorem:mainthm} and Corollary \ref{corollary:dsdecomp}]
    \par Let $\mathbf{f}(\mathbf{z})$ be a multivariate vector-valued function. There is a one-to-one correspondence between its direct sum decompositions and the orthogonal idempotent decompositions of the identity of its center.
\end{theorema}

\par Furthermore, by observing that direct sum and direct product functions can be transformed into each other via taking exponential and logarithm, we immediately obtain the following results for direct product decompositions of multivariate functions and factorizations of multivariate polynomials.

\begin{theoremb}[see Corollary \ref{corollary:dpdecomp}]
    \par Let $\mathbf{f}(\mathbf{z})$ be a multivariate vector-valued function. There is a one-to-one correspondence between its direct product decompositions and the orthogonal idempotent decompositions of the identity in the center of $\log \mathbf{f}(\mathbf{z})$.
\end{theoremb}

\begin{theoremc}[see Corollary \ref{corollary:powerprod}] 
\par A higher degree form $f(x_1, \cdots, x_n)$ is a product of independent linear forms if and only if $\cen(f)$ is a $\mathbb{C}$-subalgebra of $\mathbb{C}^{n \times n}$ that is isomorphic to $\mathbb{C}^n$.
\end{theoremc}


\subsection{Organization of the paper}
\par In Section 2, we introduce the basic notions of direct sum and direct product functions. In Section 3, we introduce the centers of multivariate vector-valued functions, and provide criteria and algorithms for direct sum and direct product decompositions. Section 4 is devoted to the application to the polynomial factorization, in particular the Chow variety of products of linear forms. 

\section{Direct sum and direct product functions}
In this section, we introduce the notions of direct sum and direct product for a set of multivariate functions. The degeneracy of multivariate functions is also introduced which leads to a natural class of examples of direct sum decompositions.

\subsection{Definitions} 

\begin{definition}
    \par An $n$-variate vector-valued function $\mathbf{f}(\mathbf{z}) = \left(f_1(\mathbf{z}), \cdots, f_m(\mathbf{z})\right)$ is called a \emph{direct sum} if there exists an \emph{invertible linear substitution} $\mathbf{z} = A\mathbf{w}$ such that
    \begin{equation}
        \label{equation:directsum}
        \begin{aligned}
            \mathbf{f}(A\mathbf{w}) 
            &= \tilde{\mathbf{f}}_1(\mathbf{w}_1) + \tilde{\mathbf{f}}_2(\mathbf{w}_2) \\
            &= \left(\tilde{f}_{11}(\mathbf{w}_1), \cdots, \tilde{f}_{1m}(\mathbf{w}_1)\right) + \left(\tilde{f}_{21}(\mathbf{w}_2), \cdots, \tilde{f}_{2m}(\mathbf{w}_2)\right) \\
            &= \left(\tilde{f}_{11}(\mathbf{w}_1) + \tilde{f}_{21}(\mathbf{w}_2), \cdots, \tilde{f}_{1m}(\mathbf{w}_1) + \tilde{f}_{2m}(\mathbf{w}_2)\right),
        \end{aligned}
    \end{equation}
    where $\mathbf{w}$ is the disjoint union of $\mathbf{w}_1$ and $\mathbf{w}_2$, both of which are nonempty. Similarly, $\mathbf{f}(\mathbf{z})$ is called a \emph{direct product} if there exists an invertible linear substitution $\mathbf{z} = A\mathbf{w}$ such that
    \begin{equation}
        \label{equation:directproduct}
        \begin{aligned}
            \mathbf{f}(A\mathbf{w}) 
            &= \tilde{\mathbf{f}}_1(\mathbf{w}_1) \odot \tilde{\mathbf{f}}_2(\mathbf{w}_2) \\
            &= \left(\tilde{f}_{11}(\mathbf{w}_1), \cdots, \tilde{f}_{1m}(\mathbf{w}_1)\right) \odot \left(\tilde{f}_{21}(\mathbf{w}_2), \cdots, \tilde{f}_{2m}(\mathbf{w}_2)\right) \\
            &= \left(\tilde{f}_{11}(\mathbf{w}_1) \cdot \tilde{f}_{21}(\mathbf{w}_2), \cdots, \tilde{f}_{1m}(\mathbf{w}_1) \cdot \tilde{f}_{2m}(\mathbf{w}_2)\right),
        \end{aligned}
    \end{equation}
    where $\mathbf{w}$ is the disjoint union of nonempty $\mathbf{w}_1$ and $\mathbf{w}_2$.
\end{definition}

\begin{remark}
    \label{remark:dstodp}
    \par There is a one-to-one correspondence between direct sum functions and direct product functions. Indeed, let $\mathbf{f}(\mathbf{z}) = \left(f_1(\mathbf{z}), \cdots, f_m(\mathbf{z})\right)$ be a direct product function and suppose that it can be represented as (\ref{equation:directproduct}). By taking the logarithm, we obtain a direct sum function 
    $$
        \log \mathbf{f}(\mathbf{z})
        = \log \left(\tilde{\mathbf{f}}_1(\mathbf{w}_1) \odot \tilde{\mathbf{f}}_2(\mathbf{w}_2)\right)
        = \log \tilde{\mathbf{f}}_1(\mathbf{w}_1) + \log \tilde{\mathbf{f}}_2(\mathbf{w}_2).
    $$
    Conversely, given any direct sum function $\mathbf{f}(\mathbf{z})$ and suppose that it can be represented as (\ref{equation:directsum}), by taking the exponential, we obtain a direct product function 
    $$
        \exp \mathbf{f}(\mathbf{z}) 
        = \exp \left(\tilde{\mathbf{f}}_1(\mathbf{w}_1) + \tilde{\mathbf{f}}_2(\mathbf{w}_2)\right)
        = \exp \tilde{\mathbf{f}}_1(\mathbf{w}_1) \odot \exp \tilde{\mathbf{f}}_2(\mathbf{w}_2).
    $$
\end{remark}

\subsection{Degeneracy} 
A natural class of examples of direct sum functions are those multivariate functions which are expressed in more variables than they actually need. For instance, the bivariate function $f(z_1, z_2) = \sin(z_1 + z_2) + \cos(2z_1 + 2z_2)$ is a direct sum as
$$
    f(z_1, z_2) = \left(\sin(w_1) + \cos(2w_1)\right) + 0(w_2),
$$
where $z_1 = w_1 - w_2$, $z_2 = w_2$ and $0(w_2)$ is the zero function in variable $w_2$. In other words, the variable set $\left\{z_1, z_2\right\}$ of $f$ is redundant and can be reduced to a smaller one (say $\{w_1\}$) by a change of variables. Hence, to study the reducibility of the variable set of a multivariate function, we need the notion of the number of essential variables.

\par To proceed, it is necessary to review some required preliminaries from multivariate calculus. Consider a family of vector-valued functions
$$
    \left\{
        \mathbf{f}_i(\mathbf{z}) = \left(
            f_{i1}(\mathbf{z}), \cdots, f_{im}(\mathbf{z})
        \right) : i = 1, \cdots, p
    \right\}.
$$
These $\mathbf{f}_i(\mathbf{z})$ are said to be \emph{$\mathbb{C}$-linearly independent} when the $\mathbb{C}$-linear combination
$$
    c_1 \mathbf{f}_1(\mathbf{z}) + \cdots + c_p \mathbf{f}_p(\mathbf{z}), 
    \qquad c_i \in \mathbb{C}
$$
is a zero function if and only if each $c_i = 0$. Otherwise, we say these $\mathbf{f}_i(\mathbf{z})$ \emph{$\mathbb{C}$-linearly dependent}. The \emph{rank} of the family $\{\mathbf{f}_i(\mathbf{z})\}_{i = 1}^{p}$ of functions, denoted by $\rank(\{\mathbf{f}_i\}_{i = 1}^{p})$, is defined to be $r$ if there exist $r$ $\mathbb{C}$-linearly independent functions in this family such that any other function can be expressed as a $\mathbb{C}$-linear combination of them.

\par Let $\mathbf{f}(\mathbf{z}) = \left(f_1(\mathbf{z}), \cdots, f_m(\mathbf{z})\right)$ be a vector-valued function of $n$ variables. Then its partial derivative with respect to $z_i \in \mathbf{z}$ is the vector-valued function
$$
    \frac{\partial \mathbf{f}}{\partial z_i}(\mathbf{z}) = \left(
        \frac{\partial f_1}{\partial z_i}(\mathbf{z}), 
        \cdots, 
        \frac{\partial f_m}{\partial z_i}(\mathbf{z})
    \right).
$$
Suppose that $\mathbf{g}(\mathbf{w}) = \left(g_1(\mathbf{w}), \cdots, g_n(\mathbf{w})\right)$ is an $l$-variate vector-valued function. Consider the composition $\mathbf{f}(\mathbf{g}(\mathbf{w}))$. The chain rule shows that
$$
    \left(
        \frac{\partial f_i}{\partial w_j}(\mathbf{w})
    \right)_{m \times l} 
    = \left(
        \frac{\partial f_i}{\partial z_j}(\mathbf{g})
    \right)_{m \times n} \cdot \left(
        \frac{\partial g_i}{\partial w_j}(\mathbf{w})
    \right)_{n \times l}.
$$
By this fact, we conclude the following useful result:

\begin{lemma}
    \label{lemma:change}
    \par Let $\mathbf{z} = A \mathbf{w}$ be an invertible linear substitution.
    \begin{enumerate}
        \item[\upshape(1)] If $\tilde{\mathbf{f}}(\mathbf{w}) = \left(\tilde{f}_1(\mathbf{w}), \cdots, \tilde{f}_m(\mathbf{w})\right)$ is the vector-valued function obtained from $\mathbf{f}(\mathbf{z})$ via the substitution $\mathbf{z} = A\mathbf{w}$, then
        $$
            \left(
                \frac{\partial \tilde{\mathbf{f}}}{\partial w_1}, \cdots, \frac{\partial \tilde{\mathbf{f}}}{\partial w_n}
            \right) = \left(
                \frac{\partial \mathbf{f}}{\partial z_1}, \cdots, \frac{\partial \mathbf{f}}{\partial z_n}
            \right) A.
        $$

        \item[\upshape(2)] If $\tilde{f}(\mathbf{w})$ is the scalar-valued function obtained from $f(\mathbf{z})$ via $\mathbf{z} = A\mathbf{w}$, then
        $$
            H(\tilde{f}) = A^{t} H(f) A,
        $$
        where $H(\tilde{f}) = \left(\partial^2 \tilde{f} / \partial w_i \partial w_j\right)_{n \times n}$ and $H(f) = \left(\partial^2 f / \partial z_i \partial z_j \right)_{n \times n}$ denote the Hessian matrices of $\tilde{f}$ and $f$, respectively.
    \end{enumerate}
\end{lemma}

\begin{proof}
    \par If we consider $\mathbf{z} = A \mathbf{w}$ as a function in $\mathbf{w}$, then due to the chain rule we have
    $$
        \left(
            \frac{\partial f_i}{\partial w_j}(\mathbf{w})
        \right)_{m \times n} 
        = \left(
            \frac{\partial f_i}{\partial z_j}(\mathbf{z})
        \right)_{m \times n} \cdot \left(
            \frac{\partial z_i}{\partial w_j}(\mathbf{w})
        \right)_{n \times n}
        = \left(
            \frac{\partial f_i}{\partial z_j}(\mathbf{z})
        \right)_{m \times n} \cdot A.
    $$
    Note that the $j$th column of the matrix $\left(\partial f_i / \partial w_j \right)_{m \times n}$ is exactly the function $\partial \mathbf{f} / \partial w_m$ (written as a column vector). By comparing each identity, we obtain (1).

    \par As for (2), it follows from the chain rule again that
    $$
        \frac{\partial^2 \tilde{f}}{\partial w_i \partial w_j}
        = \sum_{k = 1}^{n} \sum_{l = 1}^{n} \frac{\partial^2 f}{\partial z_k \partial z_l} a_{ki} a_{lj}.
    $$
    A direct verification yields the required identity.
\end{proof}

\par Now, we define the number of essential variables of a given function as follows:

\begin{definition}
    \par Let $\mathbf{f}(\mathbf{z}) = \left(f_1(\mathbf{z}), \cdots, f_m(\mathbf{z})\right)$ be an $n$-variate function. The \emph{number of essential variables} $\ess(\mathbf{f})$ of $\mathbf{f}(\mathbf{z})$ is defined to be
    $$
        \ess(\mathbf{f}) 
        = \rank \left\{
            \frac{\partial \mathbf{f}}{\partial z_1}, \cdots, \frac{\partial \mathbf{f}}{\partial z_n}
        \right\}.
    $$
    Furthermore, a function $\mathbf{f}(\mathbf{z})$ is said to be \emph{degenerate} if
    $$
        \ess(\mathbf{f}) < \card(\mathbf{z}).
    $$
    Otherwise, $\mathbf{f}(\mathbf{z})$ is said to be \emph{nondegenerate}.
\end{definition}

\begin{remark}
    \par Suppose that $\tilde{\mathbf{f}}(\mathbf{w})$ is the function obtained from $\mathbf{f}(\mathbf{w})$ by an invertible change of variables $\mathbf{z} = A \mathbf{w}$. By Lemma \ref{lemma:change}, we have
    $$
        \rank \left\{\frac{\partial \tilde{\mathbf{f}}}{\partial w_1}, \cdots, \frac{\partial \tilde{\mathbf{f}}}{\partial w_n}\right\}
        = \rank \left\{\frac{\partial \mathbf{f}}{\partial z_1}, \cdots, \frac{\partial \mathbf{f}}{\partial z_n}\right\}.
    $$
    Therefore, the number of essential variables is invariant under invertible linear changes of variables.
\end{remark}

\begin{proposition}
    \label{proposition:reduce}
    \par Let $\mathbf{f}(\mathbf{z})$ be a multivariate vector-valued function. Then either it is nondegenerate or it is a direct sum as 
    $$
        \mathbf{f}(\mathbf{z}) 
        = \mathbf{f}(A \mathbf{w}) 
        = \tilde{\mathbf{f}}_1(\tilde{\mathbf{w}}) + \mathbf{0}(\mathbf{w} \setminus \tilde{\mathbf{w}}),
    $$
    where $\tilde{\mathbf{w}} \subseteq \mathbf{w}$ and the minimal possible $\card(\tilde{\mathbf{w}})$ is $\ess(\mathbf{f})$.
\end{proposition}

\begin{proof}
    \par Suppose that $\ess(\mathbf{f}) = r$. Without loss of generality, we may further suppose that $\{\partial \mathbf{f} / \partial z_1, \cdots, \partial \mathbf{f} / \partial z_r\}$ is a linearly independent set and
    $$
        \frac{\partial \mathbf{f}}{\partial z_j} 
        = a_{j1} \frac{\partial \mathbf{f}}{\partial z_1} + \cdots + a_{jr} \frac{\partial \mathbf{f}}{\partial z_r}, \qquad j = r + 1, \cdots, n.
    $$
    Consider the following invertible matrix
    $$
        A = \left(\begin{array}{cccccccc}
            1 & 0 & \cdots & 0 & a_{(r + 1) 1} & a_{(r + 2) 1} & \cdots & a_{n1} \\
            0 & 1 & \cdots & 0 & a_{(r + 1) 2} & a_{(r + 2) 2} & \cdots & a_{n2} \\
            \vdots & \vdots & & \vdots & \vdots & \vdots & & \vdots \\
            0 & 0 & \cdots & 1 & a_{(r + 1) r} & a_{(r + 2) r} & \cdots & a_{nr} \\
            0 & 0 & \cdots & 0 & -1 & 0 & \cdots & 0 \\
            0 & 0 & \cdots & 0 & 0 & -1 & \cdots & 0 \\
            \vdots & \vdots & & \vdots & \vdots & \vdots & & \vdots \\
            0 & 0 & \cdots & 0 & 0 & 0 & \cdots & -1
        \end{array}\right).
    $$
    By $\tilde{\mathbf{f}}(\mathbf{w})$ we denote the function obtained from $\mathbf{f}(\mathbf{z})$ by the change of variables $\mathbf{z} = A\mathbf{w}$. Then Lemma \ref{lemma:change} shows that
    $$
        \frac{\partial \tilde{\mathbf{f}}}{\partial w_j} 
        = a_{j1}\frac{\partial \mathbf{f}}{\partial z_1} + \cdots + a_{jr}\frac{\partial \mathbf{f}}{\partial z_r} - \frac{\partial \mathbf{f}}{\partial z_j} = 0, \qquad j = r + 1, \cdots, n,
    $$
    and $\left\{\partial \tilde{\mathbf{f}} / \partial w_1, \cdots, \partial \tilde{\mathbf{f}} / \partial w_r\right\}$ is a linearly independent set. Hence, the function $\tilde{\mathbf{f}}(\mathbf{w})$ can be written as
    $$
        \tilde{\mathbf{f}}(\mathbf{w}) = \tilde{\mathbf{f}}_1(w_1, \cdots, w_r) + \mathbf{0}(w_{r + 1}, \cdots, w_n)
    $$
    where $\mathbf{0}(w_{r + 1}, \cdots, w_n)$ is a zero function.

    \par Finally, we prove that $r$ is the minimum. Assume to the contrary that $\mathbf{f}(\mathbf{z})$ is also a direct sum as
    $$
        \mathbf{f}(\mathbf{z}) = \mathbf{f}(B \mathbf{u}) = \bar{\mathbf{f}}_1(\bar{\mathbf{u}}) + \mathbf{0}(\mathbf{u} \setminus \bar{\mathbf{u}}),
    $$
    where $\card(\mathbf{u}) < \card(\mathbf{w}) = r$. Denoting $\bar{\mathbf{f}}(\mathbf{u}) = \mathbf{f}(B \mathbf{u})$, it is clearly seen that $\ess(\bar{\mathbf{f}}) < r$. This is a contradiction since the number of essential variables is invariant under the invertible linear substitutions.
\end{proof}


\par The proof of Proposition \ref{proposition:reduce} yields an algorithm for detecting the degeneracy of a function and calculating a suitable linear substitution to eliminate the ``redundant'' variables when the function is actually degenerate.

\begin{algorithm}
    \label{algorithm:tonondeg}
    \par Let $\mathbf{f}(\mathbf{z})$ be an $n$-variate function.
    \begin{enumerate}[leftmargin = 4em]
        \item[\sc Step 1.] (Detect the degeneracy) Calculate $r = \rank \{\partial \mathbf{f} / \partial z_1, \cdots, \partial \mathbf{f} / \partial z_n\}$. If $r = n$, then $\mathbf{f}(\mathbf{z})$ is nondegenerate and we stop. Otherwise, $\mathbf{f}(\mathbf{z})$ is degenerate and we move to the next step.
        
        \item[\sc Step 2.] (Determine the invertible linear substitution) Suppose that $\{\partial \mathbf{f} / \partial z_1, \cdots, \partial \mathbf{f} / \partial z_r\}$ is a maximal linearly independent subset of $\{\partial \mathbf{f} / \partial z_1, \cdots, \partial \mathbf{f} / \partial z_n\}$. By applying the invertible linear substitution $\mathbf{z} = A \mathbf{w}$ mentioned in the proof of Proposition \ref{proposition:reduce}, we obtain a nondegenerate function $\tilde{\mathbf{f}}(\mathbf{w})$.
    \end{enumerate}
\end{algorithm}

\begin{remark}
    \par Thanks to Algorithm \ref{algorithm:tonondeg}, any degenerate function can be transformed into a nondegenerate one. Therefore, from now on we need only focus on the nondegenerate case.
\end{remark}

\subsection{Examples} 
To elucidate Proposition \ref{proposition:reduce} and Algorithm \ref{algorithm:tonondeg}, we revisit two classical problems in linear algebra and analytic geometry. One is to determine when a quadratic form is a product of two linear forms, the other is when a surface is a cylinder.

\begin{example} \label{example: factoring quadratic forms}
Let $Q(\mathbf{z})=\mathbf{z}^tA\mathbf{z}$ be a nonzero $n$-variate complex quadratic form with Gram matrix $A$. To avoid the trivial situation, we assume $n>2$. It is not hard to see that $\ess(Q)=\rank A.$ If $Q$ can be represented as a product $L_1(\mathbf{z})L_2(\mathbf{z})$ of two linear forms in $\mathbf{z}$, then it can be expressed in at most two variables and obviously $\ess(Q) \le 2$, and so $Q$ is a direct sum. Conversely, if $\ess(Q) \le 2$, then by Proposition \ref{proposition:reduce} $Q$ is a direct sum of a quadratic form in at most two variables and a zero quadratic form, and in this case $Q$ is obviously a produt two linear forms. 
\end{example}

\begin{example} \label{example: cylinder}
    Consider a surface in three-dimensional Euclidean space given by $f(x,y,z)=0$. Then it is a cylinder if and only if $f(x,y,z)$ is degenerate, that is, $\partial f / \partial x, \partial f / \partial y, \partial f / \partial z$ are linearly dependent. For example, let 
    $$
        f(x,y,z)=5x^2+5y^2+2z^2-8xy-2xz-2yz+20x+20y-40z-16
    $$
    It is ready to see that 
    $$
        \frac{\partial f}{\partial x}+\frac{\partial f}{\partial y}+\frac{\partial f}{\partial z}=0.
    $$    
    According to Algorithm \ref{algorithm:tonondeg}, take the change of variables $x=u-w, y=v-w, z=-w$ and get 
    $$
        f=5u^2-8uv+5v^2+20u+20v-16.
    $$
    It follows that this is an elliptic cylinder.
\end{example}

\section{Centers and decompositions of multivariate functions} 
\label{section:centers}
In this section, we introduce an invariant algebra, the so-called center, to tackle the decompositions problem of multivariate vector-valued functions. We provide criteria and algorithms for both direct sum and direct product decompositions.

\subsection{The center of a multivariate vector-valued function} To begin with, we briefly recall the related theories of the center of a set of multivariate polynomials. Let $\left\{f_1(\mathbf{z}), \cdots, f_m(\mathbf{z})\right\}$ be a set of $n$-variate polynomials. Its center is defined as
$$
    \cen(f_1, \cdots, f_m) = \left\{
        X \in \mathbb{C}^{n \times n} : 
        \text{$(H(f_k)X)^t = H(f_k)X$ for all $1 \le k \le m$}
    \right\}.
$$
Here, $H(f_k)$ denotes the Hessian matrix of $f_k$. The main result of this theory is a criterion for decomposition, that is, there is a one-to-one correspondence between the simultaneous direct sum decompositions of $f_1, \cdots, f_m$ and orthogonal idempotent decompositions of the center $\cen(f_1, \cdots, f_m)$, see \cite[Theorem 3.1]{Fang}. These can be extended to all multivariate functions with second-order partial derivatives without barrier.

\begin{definition}
    \par Let $\mathbf{f}(\mathbf{z}) = \left(f_1(\mathbf{z}), \cdots, f_m(\mathbf{z})\right)$ be an $n$-variate function. The center of $\mathbf{f}(\mathbf{z})$ is defined to be 
    $$
        \cen(\mathbf{f}) = \left\{X \in \mathbb{C}^{n \times n} : \text{$(H(f_k) X)^t = H(f_k) X$ for all $1 \le k \le m$}\right\},
    $$
    where $H(f_k)$ is the Hessian matrix of the component $f_k$.
\end{definition}

\begin{remark}
    \par The center $\cen(\mathbf{f})$ is a $\mathbb{C}$-subspace but not necessarilly a $\mathbb{C}$-subalgebra of the full matrix algebra $\mathbb{C}^{n \times n}$ since $\cen(\mathbf{f})$ may not be closed under the matrix multiplication. Indeed, $\cen(\mathbf{f})$ is a Jordan algebra  \cite{Mcc} whose multiplication $*$ is:
    $$
        * : \cen(\mathbf{f}) \times \cen(\mathbf{f}) \to \cen(\mathbf{f}),
        \qquad (x, y) \mapsto \frac{x y + y x}{2}.
    $$
\end{remark}

\begin{remark}
    \label{remark:center}
    \par Suppose that $\tilde{\mathbf{f}}(\mathbf{w}) = \left(\tilde{f}_1(\mathbf{w}), \cdots, \tilde{f}_m(\mathbf{w})\right)$ is the $n$-variate function obtained from $\mathbf{f}(\mathbf{z})$ by the invertible linear substitution $\mathbf{z} = A \mathbf{w}$. It follows from Lemma \ref{lemma:change} that 
    $$
        H(\tilde{f}_k) = A^t H(f_k) A, \qquad k = 1, \cdots, m.
    $$
    Hence, we have
    $$
        \begin{aligned}
            \cen(\tilde{\mathbf{f}}) &= \left\{Y \in \mathbb{C}^{n \times n} : \text{$(H(\tilde{f}_k)Y)^t = H(\tilde{f}_k)Y$ for all $1 \le k \le m$}\right\} \\
            &= \left\{Y \in \mathbb{C}^{n \times n} : \text{$(A^t H(f_k) AY)^t = A^t H(f_k) A Y$ for all $1 \le k \le m$}\right\} \\
            &= \left\{Y \in \mathbb{C}^{n \times n} : \text{$(H(f_k) AYA^{-1})^t = H(f_k) A Y A^{-1}$ for all $1 \le k \le m$}\right\} \\
            &= \left\{Y \in \mathbb{C}^{n \times n} : AYA^{-1} \in \cen(\mathbf{f})\right\} \\
            &= A^{-1}\cen(\mathbf{f})A.
        \end{aligned}
    $$
\end{remark}

\subsection{Direct sum decompositions for vector-valued functions} By utilizing the center of functions, we are able to deal with the direct sum decompositions of multivariate vector-valued functions.

\begin{lemma}
    \label{lemma:separate}
    \par Let $f(\mathbf{z})$ be an $n$-variate scalar-valued function. Then, it can be represented as the following form 
    $$
        f(\mathbf{z}) = f_1(\mathbf{z} \setminus \{z_i\}) + f_2(\mathbf{z} \setminus \{z_j\})
    $$
    if and only if $\frac{\partial^2 f(\mathbf{z})}{\partial z_i \partial z_j} = 0$.
\end{lemma}

\begin{proof}
    \par The ``only if'' part follows immediately by taking the partial derivative of $f(\mathbf{z})$ with respect to $z_i$ and $z_j$ successively. As for the ``if'' part, by integrating $\partial^2 f / \partial z_i \partial z_j$ with respect to $z_j$, we obtain
    $$
        \frac{\partial f(\mathbf{z})}{\partial z_i} = c_0(\mathbf{z} \setminus \{z_j\})
    $$
    for some $(n - 1)$-variate function $c_0$. Then, by taking the integral of $\partial f / \partial z_i$ with respect to $z_i$, we have
    $$
        f(\mathbf{z}) = \tilde{c}_0(\mathbf{z} \setminus \{z_j\}) + c_1(\mathbf{z} \setminus \{z_i\})
    $$
    for some $(n - 1)$-variate functions $\tilde{c}_0$ and $c_1$.
\end{proof}

\par The following theorem provides a criterion for determining whether a multivariate vector-valued function is a direct sum.

\begin{theorem}
    \label{theorem:mainthm}
    \par Let $\mathbf{f}(\mathbf{z}) = \left(f_1(\mathbf{z}), \cdots, f_m(\mathbf{z})\right)$ be an $n$-variate function, and let $A \in \gl_n(\mathbb{C})$. Then the following statements are equivalent.
    \begin{enumerate}
        \item[\upshape(1)] By the linear change of variables $\mathbf{z} = A \mathbf{w}$, we have
        $$
            \mathbf{f}(A\mathbf{w}) := \tilde{\mathbf{f}}(\mathbf{w}) = \tilde{\mathbf{f}}(\mathbf{w}_1) + \tilde{\mathbf{f}}(\mathbf{w}_2),
        $$
        where $\mathbf{w} = \mathbf{w}_1 \sqcup \mathbf{w}_2$ and $\card(\mathbf{w}_j) = n_j$.

        \item[\upshape(2)] For each $k \in \{1, \cdots, m\}$, we have
        $$
            A^{t} H(f_k) A = \left(\begin{array}{cc}
                G_{k1} & \\ & G_{k2}
            \end{array}\right),
        $$
        where $G_{kj}$ is an $n_j \times n_j$ matrix.

        \item[\upshape(3)] The identity of $\cen(\mathbf{f})$ is a sum $e_1 + e_2$ of a pair of orthogonal idempotents in $\cen(\mathbf{f})$, and
        $$
            A^{-1} e_1 A = \left(\begin{array}{cc}
                \id_{n_1} & \\ & 0
            \end{array}\right) \quad \text{and} \quad
            A^{-1} e_2 A = \left(\begin{array}{cc}
                0 & \\ & \id_{n_2}
            \end{array}\right).
        $$
    \end{enumerate}
\end{theorem}

\begin{proof}
    \par  By $\tilde{\mathbf{f}}(\mathbf{w}) = (\tilde{f}_1(\mathbf{w}), \cdots, \tilde{f}_m(\mathbf{w}))$ we denote the function obtained from $\mathbf{f}(\mathbf{z})$ by the substitution $\mathbf{z} = A \mathbf{w}$. Lemma \ref{lemma:change} shows that 
    $$
        H(\tilde{f}_k) = A^t H(f_k) A, \qquad k = 1, \cdots, m.
    $$ 
    
    \par (1) implies (2). Since each component $\tilde{f}_k(\mathbf{w})$ of $\tilde{\mathbf{f}}(\mathbf{w})$ is separable as $\tilde{f}_{k1}(\mathbf{w}_1) + \tilde{f}_{k2}(\mathbf{w}_2)$, the Hessian matrix of $\tilde{f}_k$ is block-diagonal as
    $$
        H(\tilde{f}_k) = \left(\begin{array}{cc}
            G_{k1} & \\ & G_{k2}
        \end{array}\right),
    $$
    where $G_{kj}$ is a square matrix of order $n_j$.

    \par (2) implies (3). The assumption of (2) implies that for each $k \in \{1, \cdots, m\}$,
    $$
        \varepsilon_1 = \left(\begin{array}{cc}
            \id_{n_1} & \\ & 0
        \end{array}\right), \quad
        \varepsilon_1 = \left(\begin{array}{cc}
            0 & \\ & \id_{n_2}
        \end{array}\right) \in \cen(\tilde{\mathbf{f}})
    $$
    By Remark \ref{remark:center}, $e_1 = A \varepsilon_1 A^{-1}$ and $e_2 = A \varepsilon_2 A^{-1}$ is a pair of orthogonal idempotents in $\cen(\mathbf{f})$ such that $\id_n = e_1 + e_2$.

    \par (3) implies (1). For each component $\tilde{f_k}(\mathbf{w})$ of $\tilde{\mathbf{f}}(\mathbf{w})$, we have
    $$
        \left(H(\tilde{f}_k) \varepsilon_i\right)^t = H(\tilde{f}_k) \varepsilon_i, 
        \quad i = 1, 2.
    $$
    This shows that $H(\tilde{f}_k)$ is block-diagonal as in condition (2). We introduce the notation $\alpha(\mathbf{w})$ to denote a function in the variable set $\mathbf{w}$. By Lemma \ref{lemma:separate}, we have
    $$
        \frac{\partial^2 \tilde{f}_k}{\partial w_1 \partial w_{n_1 + 1}} = 0
        \implies \tilde{f}_k(\mathbf{w}) 
        = \alpha(\mathbf{w} \setminus \{w_1\}) + \alpha(\mathbf{w} \setminus \{w_{n_1 + 1}\}).
    $$
    The proof requires two successive applications of induction. For the first one, suppose that we have shown
    $$
        \tilde{f}_k(\mathbf{w}) 
        = \alpha(\mathbf{w} \setminus \{w_1\}) + \alpha(\mathbf{w} \setminus \{w_{n_1 + 1}, \cdots, w_{n_1 + a}\}),
    $$
    Then $\partial^2 \tilde{f}_k / \partial w_1 \partial w_{n_1 + a + 1}$ = 0 implies that
    $$
        \begin{aligned}
            \tilde{f}_k(\mathbf{w}) 
            &= \alpha(\mathbf{w} \setminus \{w_1\}) + \alpha(\mathbf{w} \setminus \{w_{n_1 + 1}, \cdots, w_{n_1 + a}\}) \\ 
            &= \alpha(\mathbf{w} \setminus \{w_1\}) + \alpha(\mathbf{w} \setminus \{w_1, w_{n_1 + 1}, \cdots, w_{n_1 + a}\}) + \alpha(\mathbf{w} \setminus \{w_{n_1 + 1}, \cdots, w_{n_1 + a + 1}\}) \\
            &= \alpha(\mathbf{w} \setminus \{w_1\}) + \alpha(\mathbf{w} \setminus \{w_{n_1 + 1}, \cdots, w_{n_1 + a + 1}\}) \\
        \end{aligned}
    $$
    The second induction is based on the situation
    $$
        \tilde{f}_k(\mathbf{w}) = \alpha(\mathbf{w} \setminus \{w_1\}) + \alpha(\mathbf{w} \setminus \mathbf{w}_2).
    $$
    Similarly, we will prove $\tilde{f}_k(\mathbf{w}) = \alpha(\mathbf{w} \setminus \mathbf{w}_1) + \alpha(\mathbf{w} \setminus \mathbf{w}_2) = \alpha(\mathbf{w}_2) + \alpha(\mathbf{w}_1)$.
\end{proof}

\begin{remark}
    \par If we require the matrix $A$ in Theorem \ref{theorem:mainthm} to be unitary, then the idempotents $e_1, e_2$ appearing in condition (3) should additionally be required to be hermitian.
\end{remark}

\begin{corollary}
    \label{corollary:dsdecomp}
    \par Let $\mathbf{f}(\mathbf{z}) = (f_1(\mathbf{z}), \cdots, f_m(\mathbf{z}))$ be an $n$-variate function, and $A \in \gl_n(\mathbb{C})$. Then the following statements are equivalent.
    \begin{enumerate}
        \item[\upshape(1)] By the linear substitution $\mathbf{z} = A \mathbf{w}$, we have
        $$
            \mathbf{f}(\mathbf{z}) = \mathbf{f}(A\mathbf{w}) = \tilde{\mathbf{f}}_1(\mathbf{w}_1) + \cdots + \tilde{\mathbf{f}}_p(\mathbf{w}_p),
        $$
        where $\card(\mathbf{w}_j) = n_j$ and each $\tilde{\mathbf{f}}_k$ is not a direct sum function.

        \item[\upshape(2)] For each $k \in \{1, \cdots, m\}$,
        $$
            A^t H(f_k) A = \left(\begin{array}{ccc}
                G_{k1} & & \\ & \ddots & \\ & & G_{kp}
            \end{array}\right),
        $$
        where $G_{kj}$ is an $n_j \times n_j$ matrix. Moreover, for each block $G_{kj}$, there exists no invertible matrix $\hat{A} \in \mathbb{C}^{n_j \times n_j}$ such that $\hat{A}^t G_{kj} \hat{A}$ is block-diagonal.

        \item[\upshape(3)] The identity of $\cen(\mathbf{f})$ is a sum $e_1 + \cdots + e_p$ of $p$ primitive orthogonal idempotents in $\cen(\mathbf{f})$, and 
        $$
            A^{-1} e_j A = \left(\begin{array}{ccc}
                \delta_{1j} \id_{n_1} & & \\
                & \cdots & \\
                & & \delta_{pj} \id_{n_p}
            \end{array}\right),
            \qquad j = 1, \cdots, p,
        $$
        where $\delta_{ij}$ is the Kronecker delta.
    \end{enumerate}
\end{corollary}

\par For a given vector-valued function, there is an algorithm for determining whether it is a direct sum and calculating a direct sum decomposition if this is the case.

\begin{algorithm}
    \label{algorithm:dsdecomp}
    \par Let $\mathbf{f}(\mathbf{z}) = (f_1(\mathbf{z}), \cdots, f_m(\mathbf{z}))$ be a vector-value function.
    \begin{enumerate}[leftmargin = 4em]
        \item[\sc Step 1.] (Compute the center) Solve the following system of matrix equations for the unknown $X \in \mathbb{C}^{n \times n}$:
        $$
            (H(f_k)X)^t = H(f_k)^t X, \qquad k = 1, \cdots, m.
        $$
        The center $\cen(\mathbf{f})$ is the solution space of the above system.
        
        \item[\sc Step 2.] (Find a non-trivial idempotent) Choose a basis $\{B_1, \cdots, B_r\}$ for the space $\cen(\mathbf{f})$. Consider the equation 
        $$
            \left(\sum_{i = 1}^{r} x_i B_i\right)^2 = \sum_{i = 1}^{r} x_i B_i
        $$
        for the unknowns $x_i \in \mathbb{C}$. If it admits a solution $(x_1, \cdots, x_r)$ such that $\sum_{i = 1}^{r} x_i B_i$ is neither $0$ nor $\id_n$, then we obtain a non-trivial idempotent $e = \sum_{i = 1}^{r} x_i B_i$ and move to the next step. Otherwise, the function $\mathbf{f}(\mathbf{z})$ is not a direct sum and we stop.

        \item[\sc Step 3.] (Decompose the function) Note that $(e, \id_n - e)$ is a pair of non-trivial orthogonal idempotents in $\cen(\mathbf{f})$. According to Theorem \ref{theorem:mainthm}(3), we construct the following matrix
        $$
            A = \left(A_1^{(1)}, \cdots, A_{n_1}^{(1)}, A_1^{(2)}, \cdots, A_{n_2}^{(2)}\right) \in \gl(n, \mathbb{C}),
        $$
        where $\{A_i^{(1)}\}_{i = 1}^{n_1}$ and $\{A_i^{(2)}\}_{i = 1}^{n_2}$ are bases of the eigenspaces with respect to the eigenvalue $1$ for $e$ and $\id_n - e$. Applying the change of variables $\mathbf{z} = A \mathbf{w}$, we obtain
        $$
            \mathbf{f}(\mathbf{z}) = \mathbf{f}(A\mathbf{w}) = \tilde{\mathbf{f}}_1(\mathbf{w}_1) + \tilde{\mathbf{f}}_2(\mathbf{w}_2),
        $$
        where $\card(\mathbf{w}_k) = n_k$ and $\mathbf{w}$ is the disjoint union of $\mathbf{w}_1$ and $\mathbf{w}_2$. Then replace $\mathbf{f}(\mathbf{z})$ for the subsequent iterations by each summand $\tilde{\mathbf{f}}_i$ successively.
    \end{enumerate}
\end{algorithm}

\begin{example}
    \par Consider the vector-valued function $\mathbf{f}(z_1, z_2, z_3) = (f_1, f_2, f_3)$ where
    $$
        \begin{aligned}
            f_1 &= \frac{- z_{1}^{2} z_{2} + z_{1}^{2} z_{3} - 2 z_{1} z_{2}^{2} - z_{1} z_{2} z_{3} + 3 z_{1} z_{3}^{2} + z_{1} - 6 z_{2}^{2} z_{3} + 6 z_{2} z_{3}^{2} - z_{2} + 4 z_{3}}{- z_{1} z_{2} + z_{1} z_{3} - 3 z_{2} z_{3} + 3 z_{3}^{2}},
            \vspace{1ex} \\
            f_2 &= \frac{- z_{1}^{2} z_{2} + z_{1}^{2} z_{3} - 2 z_{1} z_{2}^{2} - z_{1} z_{2} z_{3} + 3 z_{1} z_{3}^{2} + z_{1} - 6 z_{2}^{2} z_{3} + 6 z_{2} z_{3}^{2} + z_{2} + z_{3}}{- z_{1} z_{2} + z_{1} z_{3} - 2 z_{2}^{2} + 2 z_{2} z_{3}},
            \vspace{1ex} \\
            f_3 &= \frac{- z_{1}^{2} z_{2} + z_{1}^{2} z_{3} - 2 z_{1} z_{2}^{2} - z_{1} z_{2} z_{3} + 3 z_{1} z_{3}^{2} + 2 z_{1} - 6 z_{2}^{2} z_{3} + 6 z_{2} z_{3}^{2} + 2 z_{2} + 3 z_{3}}{z_{1}^{2} + 2 z_{1} z_{2} + 3 z_{1} z_{3} + 6 z_{2} z_{3}}.
        \end{aligned}
    $$
    A direct calculation yields:
    $$
        \cen(\mathbf{f})
        = \mathbb{C} \left(\begin{array}{rrr}
            -2/3 & 0 & -2 \\
            1/3 & 0 & 1 \\
            1/3 & 0 & 1
        \end{array}\right) + \mathbb{C} \left(\begin{array}{rrr}
            5/3 & 0 & 2 \\
            -1/3 & 1 & -1 \\
            -1/3 & 0 & 0
        \end{array}\right) + \mathbb{C} \left(\begin{array}{rrr}
            -17/3 & -3 & -2 \\
            5/6 & 0 & 1 \\
            5/6 & 1 & 0
        \end{array}\right).
    $$
    Note that 
    $$
        e_1 = \left(\begin{array}{rrr}
            -2 & 0 & -6 \\
            1 & 0 & 3 \\
            1 & 0 & 3
        \end{array}\right) \quad \text{and} \quad
        e_2 = \id_3 - e_1 = \left(\begin{array}{rrr}
            3 & 0 & -6 \\
            -1 & 1 & -3 \\
            1 & 0 & -2
        \end{array}\right)
    $$
    are a pair of orthogonal idempotents in $\cen(\mathbf{f})$. Hence, we take the invertible linear substitution
    $$
        \mathbf{z} = \left(\begin{array}{rrr}
            -2 & 0 & -3 \\
            1 & 1 & 0 \\
            1 & 0 & 1
        \end{array}\right) \mathbf{w}
    $$
    where the first column is an eigenvector with respect to $1$ for $e_1$ and the remainder are for $e_2$. Under this substitution, we obtain
    $$
        \mathbf{f}(z_1, z_2, z_3) = \mathbf{g}_1(w_1) + \mathbf{g}_2(w_2, w_3)
    $$
    where 
    $$
        \mathbf{g}_1(w_1) = \left(
            \frac{1}{w_1}, w_1, \frac{1}{w_1}
        \right)
    $$
    and
    $$
        \mathbf{g}_2(w_2, w_3) = \left(
            \frac{1}{- w_2 + w_3} + 2 w_2 - 3 w_3, 
            \frac{1}{- w_2 + w_3} + \frac{1}{2 w_2 + 3 w_3},
            \frac{1}{- w_2 + w_3} - w_2 + w_3
        \right)
    $$
    Applying a similar discussion to $\mathbf{g}_2$, we finally obtain
    $$
        \mathbf{f}(z_1, z_2, z_3) 
        = \left(v_1, \frac{1}{v_1}, \frac{1}{v_1}\right) 
        + \left(\frac{1}{v_2}, v_2, \frac{1}{v_2}\right) 
        + \left(\frac{1}{v_3}, \frac{1}{v_3}, v_3\right),
    $$
    where $v_1 = z_1 + 2 z_2$, $v_2 = z_1 + 3 z_3$ and $v_3 = - z_2 + z_3$.
\end{example}
\begin{example}
    \par Consider the affine variety $V \subset \mathbb{A}^4$ defined by the following polynomials
    $$
        \begin{aligned}
            f_1(\mathbf{x}) &= x_{1}^{2} + 4 x_{1} x_{2} + 5 x_{2}^{2} - 2 x_{2} x_{3} + 2 x_{2} x_{4} + x_{3}^{2} - 2 x_{3} x_{4} + x_{4}^{2} - 2, \\
            f_2(\mathbf{x}) &= x_{2}^{2} - 2 x_{2} x_{4} + 4 x_{3}^{2} + 4 x_{3} x_{4} + 2 x_{4}^{2} - 2.
        \end{aligned}
    $$
    The center of this pair $(f_1, f_2)$ of polynomials (We can regard it as a vector valued function) admits a $\mathbb{C}$-basis as follow:
    $$
        \begin{array}{lll}
            B_1 = \left(\begin{array}{rrrr}
                1/2 & 1 & 0 & 0\\0 & 0 & 0 & 0\\0 & 0 & 0 & 0\\0 & 0 & 0 & 0
            \end{array}\right), &
            B_2 = \left(\begin{array}{rrrr}
                1 & 0 & -8 & -2\\0 & 1 & 4 & 1\\0 & 1 & 4 & 1\\0 & 0 & 0 & 0
            \end{array}\right), &
            B_3 = \left(\begin{array}{rrrr}
                3/4 & 0 & - 21/2 & 13/2\\1 & 4 & 4 & -2\\- 1/2 & 1 & 2 & 0\\1 & 0 & 0 & 0
            \end{array}\right), \\
            B_4 = \left(\begin{array}{rrrr}
                -1 & 0 & 4 & -2\\0 & -1 & -2 & 1\\0 & -1 & -1 & 0\\0 & 1 & 0 & 0
            \end{array}\right), &
            B_5 = \left(\begin{array}{rrrr}
                - 3/2 & 0 & 6 & 1\\0 & - 3/2 & -3 & - 1/2\\0 & -1 & - 5/2 & 0\\0 & 0 & 1 & 0
            \end{array}\right), &
            B_6 = \left(\begin{array}{rrrr}
                1 & 0 & 0 & 0\\0 & 1 & 0 & 0\\0 & 0 & 1 & 0\\0 & 0 & 0 & 1
            \end{array}\right).
        \end{array}
    $$
    One can check that $(B_1, \id_4 - B_1)$ is a pair of non-trivial idempotents of $\cen(f_1, f_2)$. By computing the eigenvectors of these two idempotents, we obtain a linear substitution
    $$
        \mathbf{x} = \left(\begin{array}{rrrr}
            1 & - 4/5 & - 2/5 & - 6/5 \\
            0 & 2/5 & 1/5 & 3/5 \\
            0 & - 1/5 & 2/5 & 1/5 \\
            0 & 2/5 & 1/5 & - 2/5
        \end{array}\right) \mathbf{y}.
    $$
    Under this substitution, we have
    $$
        (f_1, f_2) = (y_1^2 + y_2^2 - 2, y_3^2 + y_4^2 - 2),
    $$
    and this means that $V$ is isomorphic to a product of two circles in $\mathbb{A}^2$ as affine variety.
\end{example}


\subsection{Direct product decompositions for vector-valued functions} In Remark \ref{remark:dstodp}, we have shown that the direct product decomposition of a function is equivalent to the direct sum decomposition of its logarithmization. Hence, we have the following criterion for direct product functions, which follows immediately from Corollary \ref{corollary:dsdecomp}.

\begin{corollary}
    \label{corollary:dpdecomp}
    \par Let $\mathbf{f}(\mathbf{z}) = (f_1(\mathbf{z}), \cdots, f_m(\mathbf{z}))$ be an $n$-variate function, and $A \in \gl_n(\mathbb{C})$. Then the following statements are equivalent.
    \begin{enumerate}
        \item[\upshape(1)] The function $\mathbf{f}(\mathbf{z})$ is a direct product, that is, there exists an invertible linear substitution $\mathbf{z} = A \mathbf{w}$ such that
        $$
            \mathbf{f}(\mathbf{z}) = \mathbf{f}(A\mathbf{w}) = \tilde{\mathbf{f}}_1(\mathbf{w}_1) \odot \cdots \odot \tilde{\mathbf{f}}_p(\mathbf{w}_p),
        $$
        where $\card(\mathbf{w}_j) = n_j$ and each $\tilde{\mathbf{f}}_k$ is not a direct product function.

        \item[\upshape(2)] There exists an invertible matrix $A$ such that
        $$
            A^t H(\log f_k) A = \left(\begin{array}{ccc}
                G_{k1} & & \\ & \ddots & \\ & & G_{kp}
            \end{array}\right),
            \qquad k = 1, \cdots, m,
        $$
        where $G_{kj}$ is an $n_j \times n_j$ matrix. Moreover, for each block $G_{kj}$, there exists no invertible matrix $\hat{A} \in \mathbb{C}^{n_j \times n_j}$ such that $\hat{A}^t G_{kj} \hat{A}$ is block-diagonal.

        \item[\upshape(3)] The identity of $\cen(\log \mathbf{f})$ is a sum $e_1 + \cdots + e_p$ of $p$ primitive orthogonal idempotents in $\cen(\log \mathbf{f})$, and there exists an invertible matrix $A$ such that
        $$
            A^{-1} e_j A = \left(\begin{array}{ccc}
                \delta_{1j} \id_{n_1} & & \\
                & \cdots & \\
                & & \delta_{pj} \id_{n_p}
            \end{array}\right),
            \qquad j = 1, \cdots, p,
        $$
        where $\delta_{ij}$ is the Kronecker delta.
    \end{enumerate}
\end{corollary}

\par The algorithm for direct product decomposition is highly similar to Algorithm \ref{algorithm:dsdecomp}. Indeed, if one wish to calculate the direct product decomposition of the function $\mathbf{f}$, it suffices to calculate the direct sum decomposition of $\log \mathbf{f}$.

\begin{example}
    \par Consider the vector-valued function $\mathbf{f}(z_1, z_2, z_3, z_4) = (f_1, f_2)$ where
    $$
        \begin{aligned}
            f_1 &= (z_{1} z_{3} + z_{1} z_{4} + 2 z_{2} z_{3} + 2 z_{2} z_{4} + 3 z_{3}^{2} + 3 z_{3} z_{4})^2 \\
            & \times \exp(z_{1} + z_{2} + 2 z_{3} - 2 z_{4}), \\
            f_2 &= (z_{1}^{2} + 3 z_{1} z_{2} + 6 z_{1} z_{3} - z_{1} z_{4} + 3 z_{2} z_{3} + 5 z_{3}^{2} - z_{3} z_{4}) \\
            & \times \exp(z_{1} z_{2} + 3 z_{1} z_{3} + 2 z_{2}^{2} + 7 z_{2} z_{3} - 2 z_{2} z_{4} + 6 z_{3}^{2} - 4 z_{3} z_{4} - z_{4}^{2}).
        \end{aligned}
    $$
    A direct computation shows that 
    $$
        \cen(\log \mathbf{f}) = \mathbb{C} \left(\begin{array}{rrrr}
            1/4 & 0 & -1/4 & 1/2 \\
            3/4 & 1 & 5/4 & 1/2 \\
            -1/4 & 0 & 1/4 & -1/2 \\
            1/4 & 0 & -1/4 & 1/2
        \end{array}\right) + \mathbb{C} \left(\begin{array}{rrrr}
            3/4 & 0 & 1/4 & -1/2 \\
            -3/4 & 0 & -5/4 & -1/2 \\
            1/4 & 0 & 3/4 & 1/2 \\
            -1/4 & 0 & 1/4 & 1/2
        \end{array}\right)
    $$
    and these two basis elements of $\cen(\log \mathbf{f})$ forms a pair of orthogonal idempotents. By calculating the eigenvectors with respect to $1$ for each idempotents, we obtain an invertible linear substitution
    $$
        \mathbf{z} = \left(\begin{array}{rrrr}
            0 & 1 & 1 & -2 \\
            1 & 0 & -2 & 1 \\
            0 & -1 & 1 & 0 \\
            0 & 1 & 0 & 1
        \end{array}\right) \mathbf{w}.
    $$
    Under this substitution, we have
    $$
        \mathbf{f}(z_1, z_2, z_3, z_4) 
        = \mathbf{f}_1(w_1, w_2) \odot \mathbf{f}_2(w_3, w_4)
    $$
    where
    $$
        \mathbf{f}_1(w_1, w_2) = \left(
            (2 w_1 - 2 w_2)^2 \exp (w_1 - 3 w_2),
            (3 w_1 - 5 w_1) \exp \left((w_1 - 3 w_2) (2w_1 - 2w_2)\right)
        \right)
    $$
    $$
        \mathbf{f}_2(w_3, w_4) = \left(
            (w_3 + w_4)^2 \exp (w_3 - 3 w_4),
            (2 w_3 - 2 w_4) \exp \left((w_3 - 3 w_4) (w_3 + w_4)\right)
        \right).
    $$
\end{example}

\section{Applications to factorization of polynomials}

\par In this section, we apply the theory of centers developed in Section \ref{section:centers} to several explicit examples on the factorization problem of multivariate polynomials. 

\subsection{Direct product decomposition of polynomials} The first application of our theory is the direct product decomposition of polynomials, while the direct sum decomposition has been discussed in some previous work, see e.g. \cite{Har-1, Fang, Huang}.

\begin{example}
    \par Consider the $6$-variate homogeneous polynomial of degree $4$:
    $$
        \begin{aligned}
            f & = x_1^3x_4 + x_1^2x_4^2 + x_1^2x_4x_5 - x_2x_1^2x_4 + x_1^2x_6^2 + 4x_1x_3^2x_4 - 4x_1x_3x_4x_6 - x_1x_4^2x_5 \\ 
            & - 2x_1x_4x_5^2 + x_2x_1x_4x_5 + 2x_1x_4x_6^2 + 2x_1x_5x_6^2 - x_2x_1x_6^2 - 4x_3^2x_4x_5 + 4x_3^2x_6^2 + 4x_3x_4x_5x_6 \\
            & - 4x_3x_6^3 - x_4x_5x_6^2 + x_6^4.
        \end{aligned}
    $$
    By a direct computation, we obtain
    $$
        \cen(\log f) = \mathbb{C} \left(\begin{array}{rrrrrr}
            1 & 0 & 0 & 0 & 0 & 0 \\
            2 & 1 & 0 & -1 & -2 & 0 \\
            0 & 0 & 1 & 0 & 0 & -1/2 \\
            0 & 0 & 0 & 0 & 0 & 0 \\
            1 & 0 & 0 & 0 & 0 & 0 \\
            0 & 0 & 0 & 0 & 0 & 0
        \end{array}\right) + \mathbb{C} \left(\begin{array}{rrrrrr}
            1 & 0 & 0 & 0 & 0 & 0 \\
            0 & 1 & 0 & 0 & 0 & 0 \\
            0 & 0 & 1 & 0 & 0 & 0 \\
            0 & 0 & 0 & 1 & 0 & 0 \\
            0 & 0 & 0 & 0 & 1 & 0 \\
            0 & 0 & 0 & 0 & 0 & 1
        \end{array}\right).
    $$
    Note that these two basis elements forms a pair of orthogonal idempotents in $\cen(\log f)$. Hence, by calculating the eigenvectors corresponding to the eigenvalue 1 of these idempotents, we obtain the invertible linear substitution
    $$
        \mathbf{x} = \left(\begin{array}{rrrrrr}
            1 & 0 & 0 & 0 & 0 & 0 \\
            0 & 0 & 1 & 0 & 2 & 1 \\
            0 & 1 & 0 & 1/2 & 0 & 0 \\
            0 & 0 & 0 & 0 & 0 & 1 \\
            1 & 0 & 0 & 0 & 1 & 0 \\
            0 & 0 & 0 & 1 & 0 & 0
        \end{array}\right) \mathbf{y}.
    $$
    Under this substitution, we have
    $$
        f = (3 y_1^2 + 4 y_2^2 - y_1 y_3)(y_4^2 - y_5 y_6).
    $$
    Moreover, each quadratic factor in the right-hand side of the above identity is irreducible by Example \ref{example: factoring quadratic forms}.
\end{example}

\begin{example}
    \par Consider the following homogeneous polynomial taken from \cite[Example 6.6]{Fed}:
    $$
        \begin{aligned}
            f & = x_1^4 + 4x_1^3x_2 + 6x_1^2x_2^2 + 4x_1x_2^3 + 2x_2^4 + 8x_1^3x_3 + 24x_1^2x_2x_3 + 24x_1x_2^2x_3 + 8x_2^3x_3 + 24x_1^2x_3^2 \\
& + 48x_1x_2x_3^2 + 24x_2^2x_3^2 + 32x_1x_3^3 + 32x_2x_3^3 + 17x_3^4 - 12x_1^3x_4 - 36x_1^2x_2x_4 - 36x_1x_2^2x_4 \\
& - 12x_2^3x_4 - 72x_1^2x_3x_4 - 144x_1x_2x_3x_4 - 72x_2^2x_3x_4 - 144x_1x_3^2x_4 - 144x_2x_3^2x_4 - 96x_3^3x_4 \\
& + 54x_1^2x_4^2 + 108x_1x_2x_4^2 + 54x_2^2x_4^2 + 216x_1x_3x_4^2 + 217x_2x_3x_4^2 + 216x_3^2x_4^2 - 108x_1x_4^3 \\
& - 108x_2x_4^3 - 216x_3x_4^3 + 82x_4^4.
        \end{aligned}
    $$
    We calculate the direct product decomposition of its associated form $A(f)$. Here, 
    $$
        \begin{aligned}
            A(f) & = 9785 x_{1}^{8} - 32316 x_{1}^{7} x_{2} - 19488 x_{1}^{7} x_{3} - 4194 x_{1}^{7} x_{4} + 26370 x_{1}^{6} x_{2}^{2} + 40920 x_{1}^{6} x_{2} x_{3} \\
            & - 9000 x_{1}^{6} x_{2} x_{4} + 8730 x_{1}^{6} x_{3}^{2} - 7200 x_{1}^{6} x_{3} x_{4} - 7395 x_{1}^{6} x_{4}^{2} - 260 x_{1}^{5} x_{2}^{3} - 25620 x_{1}^{5} x_{2}^{2} x_{3} \\
            & + 17820 x_{1}^{5} x_{2}^{2} x_{4} - 11910 x_{1}^{5} x_{2} x_{3}^{2} + 21600 x_{1}^{5} x_{2} x_{3} x_{4} + 7140 x_{1}^{5} x_{2} x_{4}^{2} - 595 x_{1}^{5} x_{3}^{3} \\ 
            & + 6480 x_{1}^{5} x_{3}^{2} x_{4} + 3120 x_{1}^{5} x_{3} x_{4}^{2} - 1800 x_{1}^{5} x_{4}^{3} + 15 x_{1}^{4} x_{2}^{4} - 180 x_{1}^{4} x_{2}^{3} x_{3} - 360 x_{1}^{4} x_{2}^{3} x_{4} \\ 
            & + 6390 x_{1}^{4} x_{2}^{2} x_{3}^{2} - 17280 x_{1}^{4} x_{2}^{2} x_{3} x_{4} + 2970 x_{1}^{4} x_{2}^{2} x_{4}^{2} - 495 x_{1}^{4} x_{2} x_{3}^{3} - 8640 x_{1}^{4} x_{2} x_{3}^{2} x_{4} \\
            & - 720 x_{1}^{4} x_{2} x_{3} x_{4}^{2} + 2880 x_{1}^{4} x_{2} x_{4}^{3} + 15 x_{1}^{4} x_{3}^{4} - 720 x_{1}^{4} x_{3}^{3} x_{4} + 1080 x_{1}^{4} x_{3}^{2} x_{4}^{2} + 1440 x_{1}^{4} x_{3} x_{4}^{3} \\ 
            & + 30 x_{1}^{4} x_{4}^{4} + 24 x_{1}^{3} x_{2}^{5} + 60 x_{1}^{3} x_{2}^{4} x_{3} + 90 x_{1}^{3} x_{2}^{4} x_{4} + 30 x_{1}^{3} x_{2}^{3} x_{3}^{2} - 60 x_{1}^{3} x_{2}^{3} x_{4}^{2} + 15 x_{1}^{3} x_{2}^{2} x_{3}^{3} \\
            & + 4320 x_{1}^{3} x_{2}^{2} x_{3}^{2} x_{4} - 2880 x_{1}^{3} x_{2}^{2} x_{3} x_{4}^{2} + 120 x_{1}^{3} x_{2} x_{3}^{4} - 1440 x_{1}^{3} x_{2} x_{3}^{2} x_{4}^{2} - 1440 x_{1}^{3} x_{2} x_{3} x_{4}^{3} \\ 
            & + 240 x_{1}^{3} x_{2} x_{4}^{4} + 12 x_{1}^{3} x_{3}^{5} + 90 x_{1}^{3} x_{3}^{4} x_{4} - 120 x_{1}^{3} x_{3}^{3} x_{4}^{2} + 120 x_{1}^{3} x_{3} x_{4}^{4} + 36 x_{1}^{3} x_{4}^{5} \\ 
            & - 12 x_{1}^{2} x_{2}^{5} x_{3} + 15 x_{1}^{2} x_{2}^{4} x_{4}^{2} - 5 x_{1}^{2} x_{2}^{3} x_{3}^{3} + 720 x_{1}^{2} x_{2}^{2} x_{3}^{2} x_{4}^{2} - 12 x_{1}^{2} x_{2} x_{3}^{5}  - 120 x_{1}^{2} x_{2} x_{3} x_{4}^{4} \\
            & + 15 x_{1}^{2} x_{3}^{4} x_{4}^{2} + 2 x_{1}^{2} x_{4}^{6}.
        \end{aligned}
    $$
    The center of the function $\log A(f)$ is spanned by
    $$
        e_1 = \left(\begin{array}{rrrr}
            1 & 0 & 0 & 0 \\
            1 & 0 & 0 & 0 \\
            2 & 0 & 0 & 0 \\
            -3 & 0 & 0 & 0 
        \end{array}\right) \quad \text{and} \quad
        e_2 = \left(\begin{array}{rrrr}
            0 & 0 & 0 & 0 \\
            -1 & 1 & 0 & 0 \\
            -2 & 0 & 1 & 0 \\
            3 & 0 & 0 & 1
        \end{array}\right),
    $$
    which is a pair of non-trivial orthogonal idempotents satisfying $e_1 + e_2 = \id_4$. By taking the substitution
    $$
        \mathbf{x} = \left(\begin{array}{rrrr}
                1 & 0 & 0 & 0 \\
                1 & 1 & 0 & 0 \\
                2 & 0 & 1 & 0 \\
                -3 & 0 & 0 & 1
        \end{array}\right) \mathbf{y}
    $$
    and applying it to $A(f)(\mathbf{x})$, we obtain
    $$
        A(f)(\mathbf{x}) = F_1(y_1) \cdot F_2(y_2, y_3, y_4),
    $$
    where $F_1(y_1) = - y_1^2$ and
    $$
        F_2(y_2, y_3, y_4) = 12 y_{2}^{5} y_{3} - 15 y_{2}^{4} y_{4}^{2} + 5 y_{2}^{3} y_{3}^{3} - 720 y_{2}^{2} y_{3}^{2} y_{4}^{2} + 12 y_{2} y_{3}^{5} + 120 y_{2} y_{3} y_{4}^{4} - 15 y_{3}^{4} y_{4}^{2} - 2 y_{4}^{6}.
    $$
\end{example}

\begin{remark}
    \par In \cite{Fed}, Fedorchuk proposed a method for solving the direct sum decomposition of a given form $f$, which was based on the calculation of the direct product decomposition of its associated form $A(f)$. However, based on our theory, we argue that there is no essential difference between these two types of decomposition. Handling the problem of direct sum decomposition in terms of associated forms may potentially introduce unnecessary complexity, as the degree of $A(f)$ is significantly higher than that of $f$.
\end{remark}

\subsection{Orbits of monomials} In this subsection, we focus on whether and how a homogeneous polynomial $\mathbf{f}(\mathbf{x})$ can be expressed as a product of powers of linearly independent linear forms
\begin{equation}
    \label{equation:powerprod}
    f(\mathbf{x}) = l_1(\mathbf{x})^{c_1} \cdots l_m(\mathbf{x})^{c_m}, 
    \qquad c_i \in \mathbb{Z}_{> 0}.
\end{equation}
The orbit $\orb(f)$ of an $n$-variate polynomial $f(\mathbf{x})$ under the action of the general linear group $\gl(n, \mathbb{C})$ is  
$$
    \orb(f) = \left\{
        f(A \mathbf{x}) : A \in \gl(n, \mathbb{C})
    \right\}.
$$
It is obvious that any homogeneous polynomial of the form (\ref{equation:powerprod}) must belong to the orbit of the monomial $x_1^{c_1} \cdots x_m^{c_m}$. This problem is related to the Chow variety. Recall that the Chow variety $\chow(1, d, n)$ of $0$-cycles in $\mathbb{P}^{n - 1}$ of degree $d$ is the  projectivization of the space of $n$-variate homogeneous polynomials of degree $d$ which are products of linear forms \cite[Chapter 4]{Gel}. Hence, the Zariski closure of the orbit $\orb(x_1^{c_1} \cdots x_n^{c_n})$ of a monomial with total degree $d$ is actually a subvariety of $\chow(1, d, n)$.

By Algorithm \ref{algorithm:tonondeg}, it suffices to consider only the nondegenerate polynomials. In this case, if an $n$-variate homogeneous polynomial of degree $d$ can be expressed as the form (\ref{equation:powerprod}), then we must have the equality $m = n$.

\begin{corollary}
    \label{corollary:powerprod}
    \par Let $f(\mathbf{x})$ be a nondegenerate $n$-variate homogeneous polynomial of degree $d$, and $A = (a_{ij}) \in \gl_n(\mathbb{C})$. Then the following statements are equivalent:
    \begin{enumerate}
        \item[\upshape(1)] $f(\mathbf{x}) = (a_{11}x_1 + \cdots + a_{1n}x_n)^{c_1} \cdots (a_{n1}x_1 + \cdots + a_{nn}x_n)^{c_n}$, where $d = c_1 + \cdots + c_n$.
        
        \item[\upshape(2)] $(A^{-1})^{t} H(\log f) A^{-1} = - \diag \left\{\frac{c_1}{(a_{11}x_1 + \cdots + a_{1n}x_n)^2}, \cdots, \frac{c_n}{(a_{n1}x_1 + \cdots + a_{nn}x_n)^2}\right\}$.

        \item[\upshape(3)] The identity of $\cen(\log f)$ admits a primitive orthogonal idempotent decomposition as 
        $$
            1_{\cen(f)} = A^{-1}E_{11}A + \cdots + A^{-1}E_{nn}A,
        $$
        where $E_{kk}$ is the matrix whose $(k, k)$-entry is $1$ and others are $0$.
    \end{enumerate}
\end{corollary}

\par The following proposition shows that when a homogeneous polynomial is a direct product of independent linear forms, its center possesses an additional algebraic structure.
\begin{proposition}
    \label{proposition:algstru}
    \par A nondegenerate $n$-variate homogeneous polynomial $f(\mathbf{x})$ of degree $d$ can be expressed as a product of powers of independent linear forms
    $$
        f(\mathbf{x}) = (a_{11}x_1 + \cdots + a_{1n}x_n)^{c_1} \cdots (a_{n1}x_1 + \cdots + a_{nn}x_n)^{c_n}, \quad d = c_1 + \cdots + c_n
    $$
    if and only if $\cen(\log f)$ is a $\mathbb{C}$-subalgebra of $\mathbb{C}^{n \times n}$ that is isomorphic to $\mathbb{C}^{n}$.
\end{proposition}

\begin{proof}
    \par By Remark \ref{remark:center} and \ref{corollary:powerprod}, we have
    $$
        A \cen(\log f) A^{-1} = \left\{
            \text{diagonal matrices in $\mathbb{C}^{n \times n}$}
        \right\}
    $$
    where $A = (a_{ij})$. Hence, $\cen(\log f)$ is isomorphic to $\mathbb{C}^n$.
\end{proof}

\par To elucidate our approach, we consider the factorization of ternary cubics. There is a classical criterion in invariant theory which says a ternary cubic is a product of linear forms if and only if it is equal to its Hessian determinant up to a scalar, see \cite{Brook} for a detailed modern treatment. The present approach provides both a criterion and an algorithm, based on elementary linear algebra.

\begin{corollary}
    \par A ternary cubic $f(x_1, x_2, x_3)$ is a product of linear forms if and only if one of the following cases occurs:
    \begin{enumerate}
        \item[\upshape(1)] $\ess(f) = 1$.
        \item[\upshape(2)] $\ess(f) = 2$.
        \item[\upshape(3)] $\ess(f) = 3$ and $\cen(\log f) \cong \mathbb{C}^3$.
    \end{enumerate}
\end{corollary}

\begin{proof}
    \par In case (1), $f$ is a cubic of some linear form, that is,
    $$
        f(x_1, x_2, x_3) = (a_1 x_1 + a_2 x_2 + a_3 x_3)^3.
    $$
    In case (2), there exists an invertible linear substitution $\mathbf{x} = A \mathbf{y}$ such that
    $$
        \begin{aligned}
            f(x_1, x_2, x_3) 
            &= \tilde{f}(y_1, y_2) = y_1^3 + b_1 y_1^2 y_2 + b_2 y_1 y_2^2 + b_3 y_2^3 \\
            &= \left(y_1 - \lambda_1 y_2\right) \left(y_1 - \lambda_2 y_2\right) \left(y_1 - \lambda_3 y_2\right),
        \end{aligned}
    $$
    where $\lambda_1, \lambda_2, \lambda_3$ are roots of the cubic equation $\tilde{f}(y_1, 1) = 0$. As for the case (3), this is a direct consequence of Proposition \ref{proposition:algstru}.
\end{proof}

\begin{example}
    \par Consider the following ternary cubic:
    $$
        f(x_1, x_2, x_3) = x_1^3 + x_2^3 + x_3^3 - 3 x_1 x_2 x_3.
    $$
    The center of the function $\log f$ is
    $$
        \cen(\log f) = \mathbb{C} \left(\begin{array}{ccc}
            1 & 0 & 0 \\ 0 & 1 & 0 \\ 0 & 0 & 1
        \end{array}\right) + 
        \mathbb{C} \left(\begin{array}{ccc}
            0 & 1 & 0 \\ 0 & 0 & 1 \\ 1 & 0 & 0
        \end{array}\right) +
        \mathbb{C} \left(\begin{array}{ccc}
            0 & 0 & 1 \\ 1 & 0 & 0 \\ 0 & 1 & 0
        \end{array}\right).
    $$
    It is straightforward to check that
    $$
        e_1 = \left(\begin{array}{rrr}
            1/3 & 1/3 & 1/3 \\ 1/3 & 1/3 & 1/3 \\ 1/3 & 1/3 & 1/3
        \end{array}\right), \quad
        e_2 = \id_3 - e_1 = \left(\begin{array}{rrr}
            2/3 & -1/3 & -1/3 \\ -1/3 & 2/3 & -1/3 \\ -1/3 & -1/3 & 2/3
        \end{array}\right)
    $$
    is a pair of orthogonal idempotents in $\cen(\log f)$, which correspond to an invertible linear substitution 
    $$
        \mathbf{x} = \left(\begin{array}{rrr}
            1 & -1 & -1 \\ 1 & 0 & 1 \\ 1 & 1 & 0
        \end{array}\right) \mathbf{y}.
    $$
    Under this substitution, we obtain
    $$
        \begin{aligned}
            f(x_1, x_2, x_3) & = 9 y_1 (y_2^2 + y_2 y_3 + y_3^2) \\
            & = y_1 (y_2 - \bar{\omega} y_3) (y_2 - \omega y_3) \\
            & = (x_1 + x_2 + x_3) (x_1 + \omega x_2 + \omega^2 x_3) (x_1 + \omega^2 x_2 + \omega x_3),
        \end{aligned}
    $$
    where $\omega = \frac{- 1 + \sqrt{3}i}{2}$ is the cubic root of unity.
\end{example}

\begin{example}
    \par Let 
    $$
        f(x_1, x_2, x_3) = x_1^3x_2 - x_1^3x_3 - 2x_1^2x_2^2 + 3x_1^2x_2x_3 - x_1^2x_3^2 + x_1x_2^3 - 3x_1x_2^2x_3 + 2x_1x_2x_3^2 + x_2^3x_3 - x_2^2x_3^2.
    $$
    The center of $\log f$ is
    $$
        \cen(\log f) = \mathbb{C} \left(\begin{array}{rrr}
    		1/2 & 0 & 1/2 \\
    		1/2 & 0 & 1/2\\
    		1/2 &  0 & 1/2
        \end{array}\right) + \mathbb{C} \left(\begin{array}{rrr}
			0 & 1/2 & -1/2 \\
			0 & 1/2 & -1/2\\
			0 & -1/2& 1/2
        \end{array}\right) + \mathbb{C} \left(\begin{array}{rrr}
    		1/2 & -1/2 & 0\\
    		-1/2 & 1/2  & 0\\
    		-1/2 &  1/2 & 0
        \end{array}\right).
    $$
    These three matrices form a complete set of primitive orthogonal idempotents of $\cen(\log f)$, and they correspond to an invertible linear substitution
    $$
        \mathbf{x} = \left(
            \begin{array}{rrr}
                1 & 1 & 1 \\
        		1 & 1 & -1\\
        		1 & -1 & -1
            \end{array}
        \right) \mathbf{y}.
    $$
    Under this substitution, we obtain
    $$
        g(\mathbf{y}) = 16 y_1 y_2 y_3^2 = (x_1 + x_3) (x_2 - x_3) (x_1 - x_2)^2.
    $$
\end{example}

\section*{Use of AI tools declaration}
The authors declare they have not used Artificial Intelligence (AI) tools in the creation of this article.

\section*{Acknowledgements}



H.-L. Huang was partially supported by the Key Program of the Natural Science Foundation of Fujian Province (Grant No. 2024J02018) and the National Natural Science Foundation of China (Grant No. 12371037). Y. Ye was partially supported by the National Key R\&D Program of China (Grant No. 2024YFA1013802), the National Natural Science Foundation of China (Grant Nos. 12131015 and 12371042), and the Innovation Program for Quantum Science and Technology (Grant No. 2021ZD0302902).

\section*{Conflict of interest}
The authors declare no conflicts of interest.

    \bibliographystyle{plain}
    \bibliography{references}
\end{document}